%% file: StochZins.tex
\documentclass[11pt,a4paper,final]{article}
\usepackage[latin1]{inputenc}
\usepackage{amsfonts}
\usepackage{amscd}
\usepackage{amsmath}
\usepackage{theorem}
\usepackage{mathrsfs}
\usepackage[dvips]{graphicx}
\usepackage{fancybox,amssymb}
\usepackage{geometry}
\usepackage[english]{babel}
\usepackage{indentfirst}

\newcommand{\R}{\mathbb{R}}
\newcommand{\N}{\mathbb{N}}

\newcommand{\mE}{\mathbb{E}}
\newcommand{\md}{\,{\rm d}}
\newcommand{\s}{\sum\limits}

\newcommand{\w}{\wedge}
\newcommand{\bq}{\begin{eqnarray*}}
\newcommand{\eq}{\end{eqnarray*}}
\newcommand{\one}{1\mkern-5mu{\hbox{\rm I}}}

\theoremstyle{break}
\theorembodyfont{}
\newtheorem{Def}{Definition}[section]
\newtheorem{Bem}[Def]{Remark}

\newtheorem{Satz}[Def]{Proposition}

\newtheorem{Bsp}[Def]{Example}
\newtheorem{out}[Def]{Outlook}

\makeatletter
\newenvironment{proof}{\noindent{\textit{Proof:}}}{%
\unskip\nobreak\hfil\penalty50\hskip1em\null\nobreak
$\Box$
\parfillskip=\z@\finalhyphendemerits=0\endgraf\bigskip}

\let\oldendBsp\endBsp
\def\endBsp{\unskip\nobreak\hfil\penalty50\hskip1em\null\nobreak\hfil%
$\blacksquare$\parfillskip=\z@\finalhyphendemerits=0\endgraf\oldendBsp}
\let\oldendBem\endBem
\def\endBem{\unskip\nobreak\hfil\penalty50\hskip1em\null\nobreak\hfil%
$\blacksquare$\parfillskip=\z@\finalhyphendemerits=0\endgraf\oldendBem}
\let\oldendout\endout
\def\endout{\unskip\nobreak\hfil\penalty50\hskip1em\null\nobreak\hfil%
$\blacksquare$\parfillskip=\z@\finalhyphendemerits=0\endgraf\oldendout}
\makeatother

\author{\small\sc  Julia Eisenberg \footnote{email: jeisenbe@fam.tuwien.ac.at \newline The research of the author was supported by the Austrian Science Fund, grant P26487}\smallskip\\\footnotesize Institute of Mathematical Methods in Economics, Vienna University of Technology.}
\date{}

\title{Deterministic Income with Deterministic and Stochastic Interest Rates}

\begin{document}
\maketitle
\begin{abstract}\noindent
We consider an individual or household endowed with an initial capital and an income, modeled as a deterministic process with a continuous drift rate. At first, we model the discounting rate as the price of a zero-coupon bond at zero under the assumption of a short rate evolving as an Ornstein-Uhlenbeck process. Then, a geometric Brownian motion as the preference function and an Ornstein-Uhlenbeck process as the short rate are taken into consideration.   
It is assumed that the primal interest of the economic agent is to maximise the cumulated value of (expected) discounted consumption from a given time up to a finite deterministic time horizon $T\in\R_+$ or, in a stochastic setting, infinite time horizon. 
We find an explicit expression for the value function and for the optimal strategy in the first two cases. In the third case, we have to apply the viscosity ansatz.

\vspace{6pt}
\noindent
\\{\bf Key words:} optimal control, Hamilton--Jacobi--Bellman equation, Vasicek model, geometric Brownian motion, interest rate
\settowidth\labelwidth{{\it 2010 Mathematical Subject Classification: }}%
                \par\noindent {\it 2010 Mathematical Subject Classification: }%
                \rlap{Primary}\phantom{Secondary}
                93B05\newline\null\hskip\labelwidth
                Secondary 49L20, 49L25
\end{abstract}                
\input{intro}
\input{Sec1}
\input{Sec2}
\input{Bsp}
\input{sec5}
\input{Sec4}
\input{unique}

\end{document}

%% file: intro.tex
\section{Introduction}
In the recent years, there appeared a big range of papers considering dividends, consumption, capital injections, where the return functions were defined as an expected discounted value with a constant positive discounting or preference rate. Confer for instance Schmidli \cite{HS}, Albrecher and Thonhauser \cite{AlbThReview}, Cox and Huang \cite{coh}, Eisenberg \cite{eis}. It is not our target to make a review of the existing literature. Therefore, we just refer to the references in the above publications.  

In the mentioned examples, the discounting rate is a constant and does not depend on time, which makes it to a preference rate, describing investment preferences of an agent in the considered model.
Indeed, it is a usual practice that economic models make an assumption of a constant and strictly positive preference rate, which implies a ``sacrifice'' of far future for present and/or near future. This fact leads to a distortion in representation of the economic processes, to say nothing about the unrealistic assumption of market idleness in the considered time period. 

One of the possible extensions of such a model is the introduction of a stochastic interest rate. 
The  stochastisation of the model can be interpreted in two ways.
The first way is to see the stochastic rate as a possibility of a macroeconomic market changing, which would influence the consumption behaviour of a sole economic agent. A suitable example provides the recent US ``Fiscal Cliff'', which is still affecting the pocket of every individual and business in the US.
The second way is to interpret the stochastics in the interest rate as uncertainty about changes in individual preferences of the economic agent. For example, a cold summer can influence the earnings of a farmer family essentially. This can lead to a 	considerable change in the ``investment behaviour'': money today can become much more preferable to money tomorrow in the years of famine compared to the years of plenty.

But what happens if we introduce a stochastic interest rate? 
In actuarial mathematics, the surplus of an insurance entity is usually modeled via a stochastic process due to the uncertainty about future system development: stochastic models approximate the real processes much better than deterministic ones.
Adding a stochastic interest rate into a model with stochastic surplus would complicate the optimization problem a lot, even if we assume the both processes to be independent.
In contrast, deterministic modeling enjoys a much greater ease of computability. Thus, to start with, in the first part of the paper we model the surplus as a deterministic process with a continuous drift function. Further, it is assumed that the discounting function is given by the price of a pure-discount bond at time zero under the spot rate evolving due to the Vasicek model. For detailed description of the bond price theory see, for instance, Brigo and Mercurio \cite[p. 58]{brigo}. 

In \cite{jgt} Eisenberg, Grandits and Thonhauser considered the problem of consumption maximization for an arbitrary drift function under a constant preference rate. There, it was possible to establish an algorithm for determination of the value function. 
In the present problem, we use a similar principle: calculate the value function and the optimal strategy in reverse order, starting at the maturity $T$.
At first, we consider the case of restricted consumption payments and then look at the unrestricted case. Since, the case with restricted payments turned out to be more complicated, we illustrate it with an example.
In a remark, we discuss the problem for an arbitrary deterministic drift function.  

In the second part of the paper we model the surplus as a deterministic process with constant drift. But the discounting function is now a stochastic process. At first, we consider the case where the consumption of the considered economic agent is linked to a stock whose price follows a geometric Brownian motion. Then, we model the short rate as an Ornstein-Uhlenbeck process with special parameters. Just in the first case, it was possible to determine the optimal strategy and the value function. In the second case we had to apply the viscosity ansatz. Also, in the second case we consider just the case with restricted consumption rates. The case with unrestricted rates has to be considered separately and will be studied in our future research. 

To the best of our knowledge, interest rate theory is an unploughed field in insurance mathematics and can open up a lot of research possibilities. Some of them are mentioned in the concluding remark.

%% file: Sec1.tex
\section{Deterministic Preference Function \label{sec:1}}
\noindent
Consider the surplus process, where the surplus rate is given by a non-negative constant $\mu$: 
\[
X_t=x+\mu t
\]
Assume, an individual or household consumes goods depending on the price of a zero-coupon bond at time zero. The short rate is a stochastic quantity and is given by a Vasicek model. Our target is to maximise the cumulated value of the discounted consumption from a given time up to a finite deterministic time horizon $T\in\R_+$.
We do not allow the consumption to cause the ruin, which means that the endpoint of our journey will be always $T$. 
The surplus process under the consumption process $C=\{c_s\}$ is
\[
X_t^C= x+\mu t-\int_0^t c_s\md s\;.
\]
We call a strategy $C=\{c_s\}$ admissible if $c_s\in [0,\xi]$ and $X_t^C\ge 0$ for all $t\in[0,T]$.
The return function corresponding to an admissible strategy $C=\{c_s\}$ is defined as
\[
V^C(t,x)=\int_t^T \mE\Big[e^{-U^r_s}\Big]c_s\md s+X^C_T\mE\Big[e^{-U^r_T}\Big]\;,
\]
where $U_s^r=\int_0^s r_u\md u$ and $\{r_s\}$ is an Ornstein-Uhlenbeck process with $r_0=r$, i.e. $\{r_s\}$ fulfils the following integral equation
\[
r_t=r e^{-at}+\tilde b(1-e^{-at})+\tilde \sigma e^{-at}\int_0^t e^{as}\md W_s\;,
\]
where $r_0=r$ is the initial value of the process, $a,\tilde\sigma>0$, $b\in\R$ are constants and $\{W_s\}$ is a standard Brownian motion. 
Here, due to Brigo and Mercurio \cite{brigo} $\mE\big[e^{-U_s^r}\big]$ denotes the price at zero of a zero-coupon bond (or pure-discount bond) with maturity $s$.
We target to maximize the expected value of discounted consumption. 
\[
V(t,x)=\sup\limits_{C} V^C(t,x)\;.
\]
The HJB equation corresponding to the problem is given by 
\[
V_t+\mu V_x+\sup\limits_{0\le c\le\xi}c\big\{\mE\big[e^{-U_t^r}\big]-V_x\big\}=0\;.
\]
In Borodin and Salminen \cite[p. 525]{bs} one finds a closed expression for $\mE[e^{-U_s^r}]$:
\begin{align*}
\mE[e^{-U_s^r}]=\exp\Big\{-bs-\frac{r-\tilde b}a\big(1-e^{-as}\big)+\frac{\tilde \sigma^2}{4a^3}\big(2as+1-(2-e^{-as})^2\big)\Big\}\;.
\end{align*}
Letting $\sigma:=\frac{\tilde \sigma}{\sqrt{2a}}$ and $b:=\tilde b-\frac{\tilde \sigma^2}{2a^2}$, we have
\begin{align}
\mE[e^{-U_s^r}]=\exp\Big\{-bs-\frac{r- b}a\big(1-e^{-as}\big)-\frac{\sigma^2}{2a^2}(1-e^{-as})^2\Big\}\;.\label{ou}
\end{align}
Let
\begin{align}
f(s):=-bs-\frac{r- b}a\big(1-e^{-as}\big)-\frac{\sigma^2}{2a^2}(1-e^{-as})^2\;.\label{f}
\end{align}
Then, the HJB equation becomes 
\begin{align}
V_t+\mu V_x+\sup\limits_{0\le c\le\xi}c\big\{e^{f(t)}-V_x\big\}=0\;.\label{hjb}
\end{align}
Depending on the parameter choice, the function $f(s)$ will have different properties.
\subsection{The Properties of $f(t)$ \label{properties}}
\begin{figure}[t]
\includegraphics[scale=0.18, bb = 0 -500 200 900]{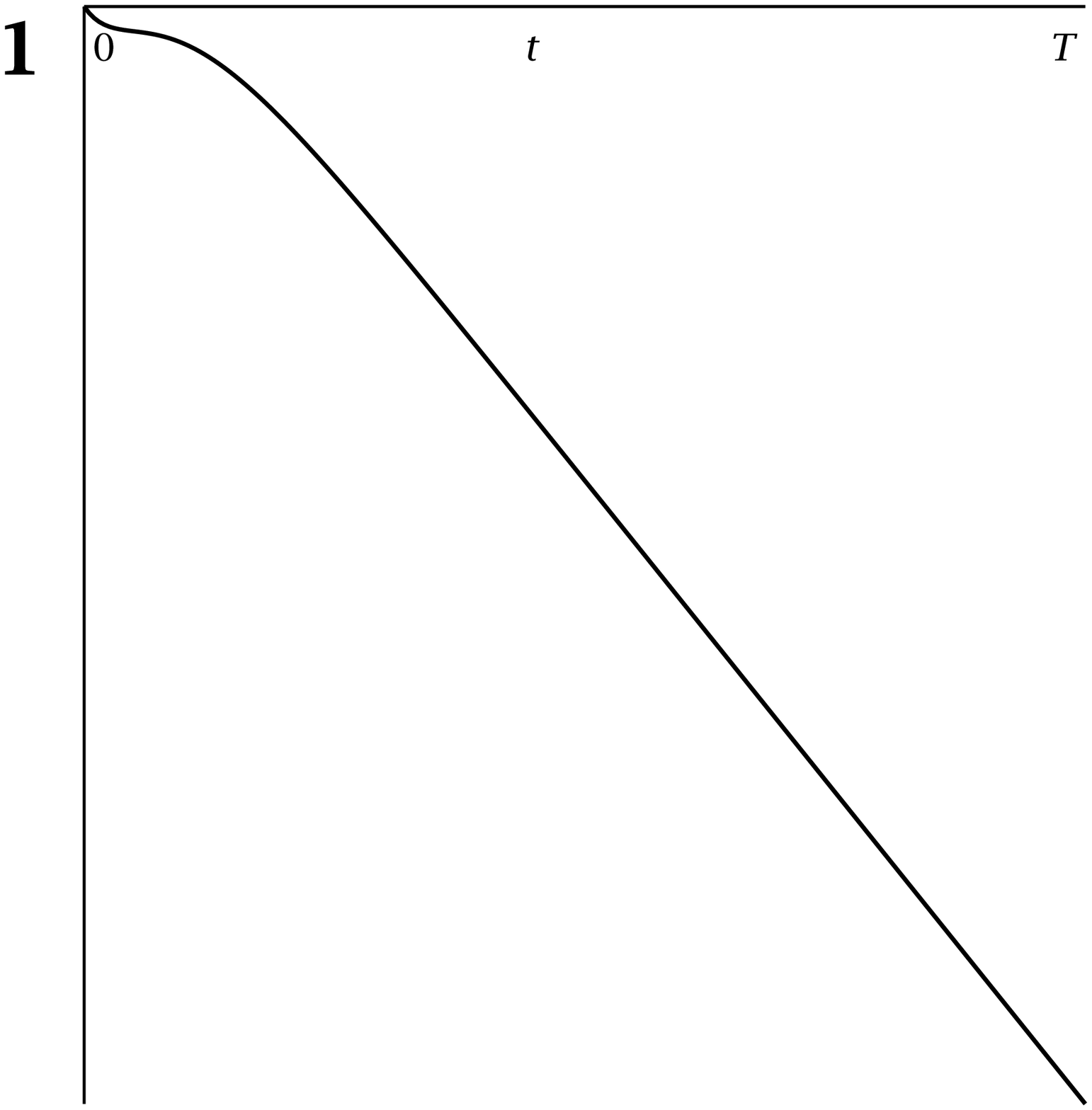}
\includegraphics[scale=0.18, bb = 220 180 400 900]{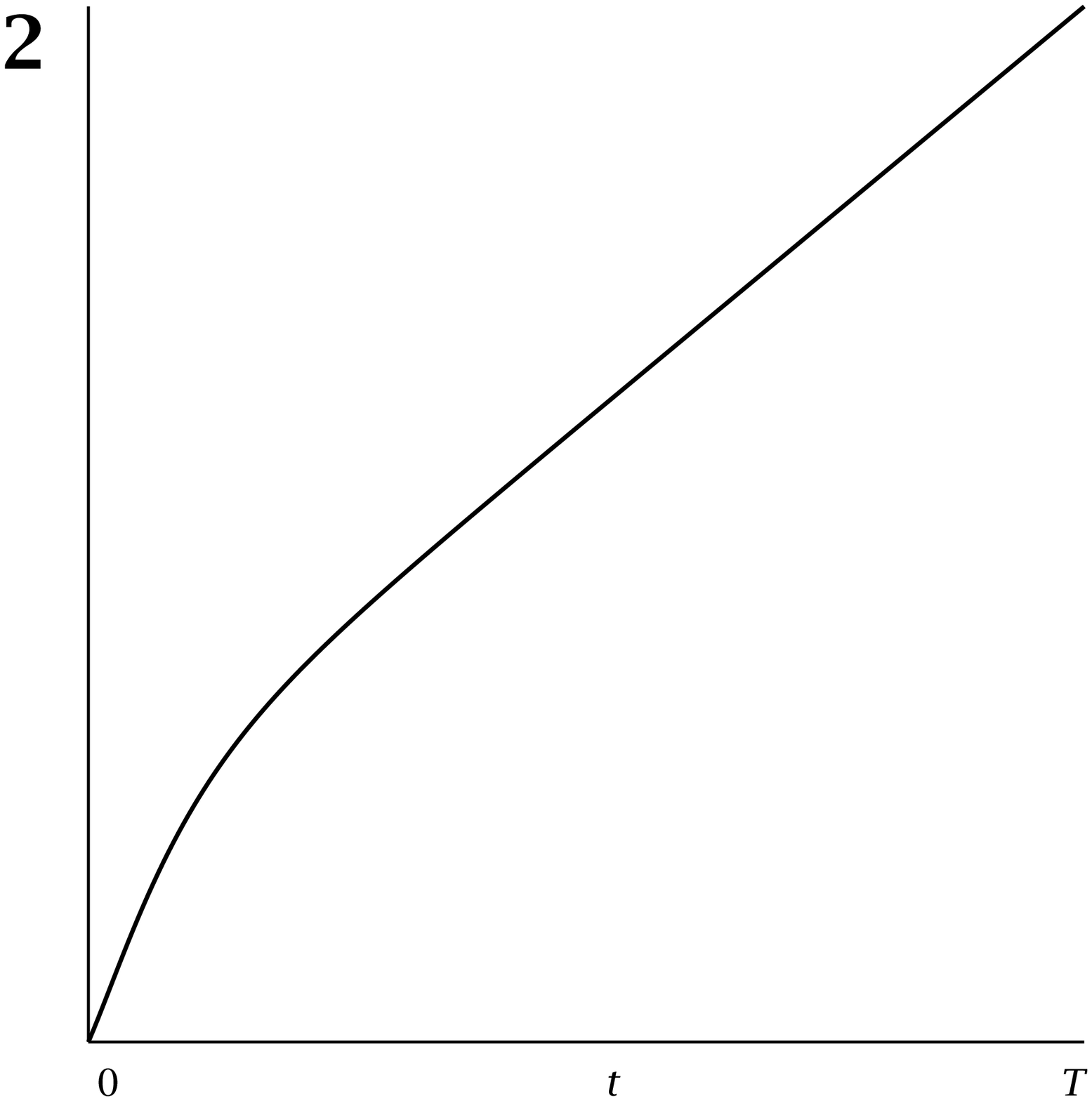}
\includegraphics[scale=0.18, bb = -350 -500 0 900]{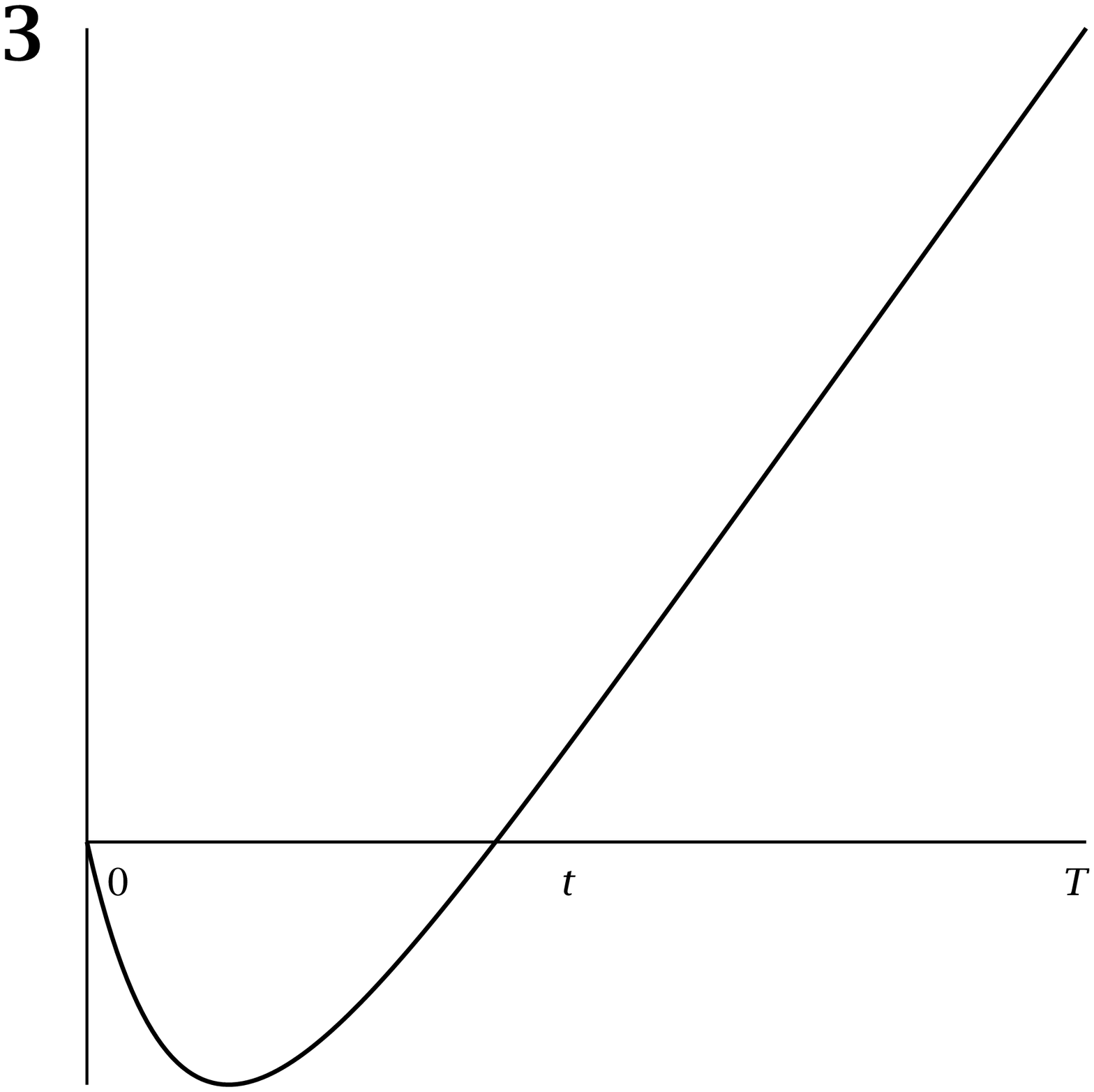}
\includegraphics[scale=0.18, bb = 20 180 200 900]{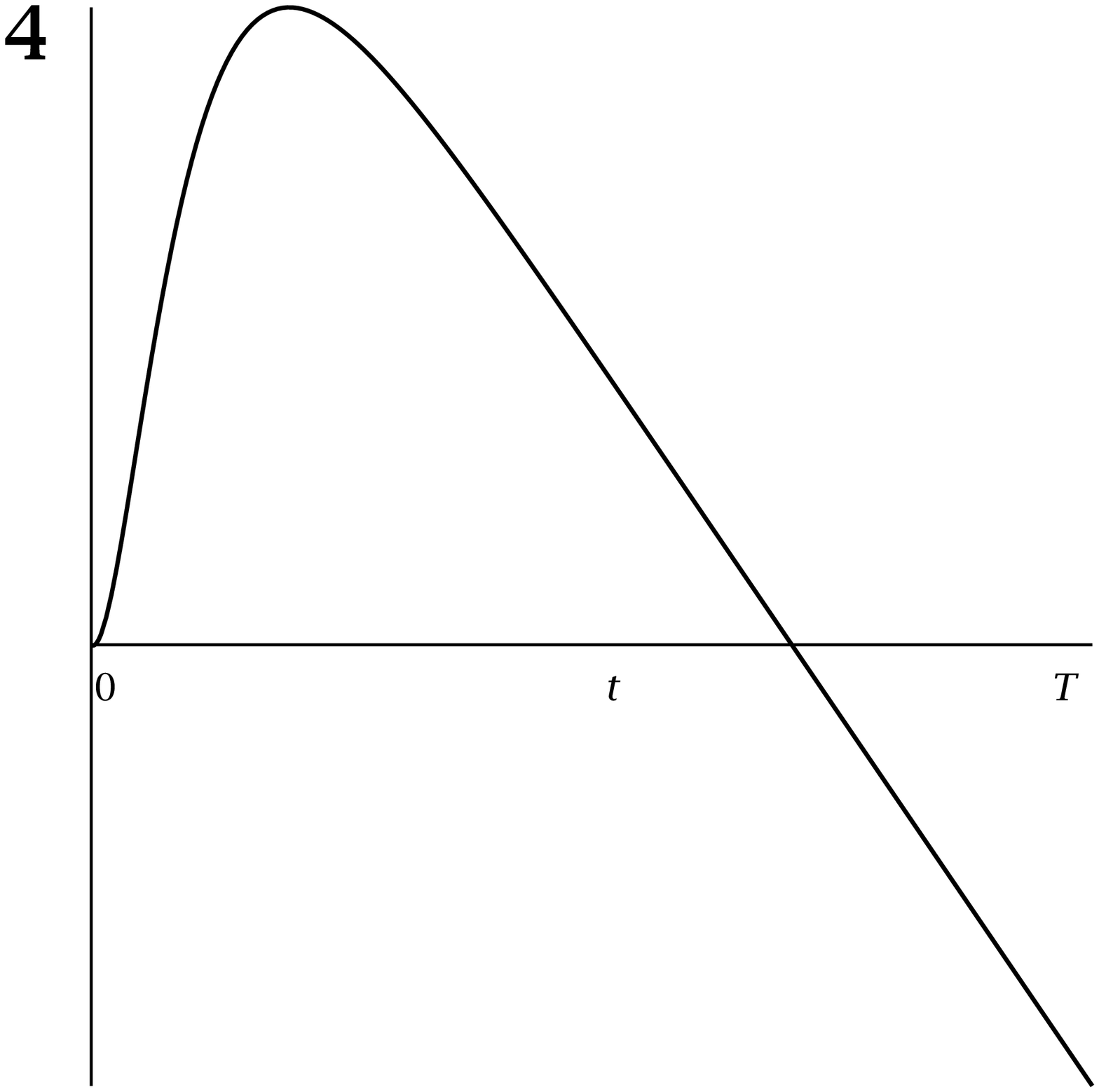}
\includegraphics[scale=0.18, bb = -500 -500 0 900]{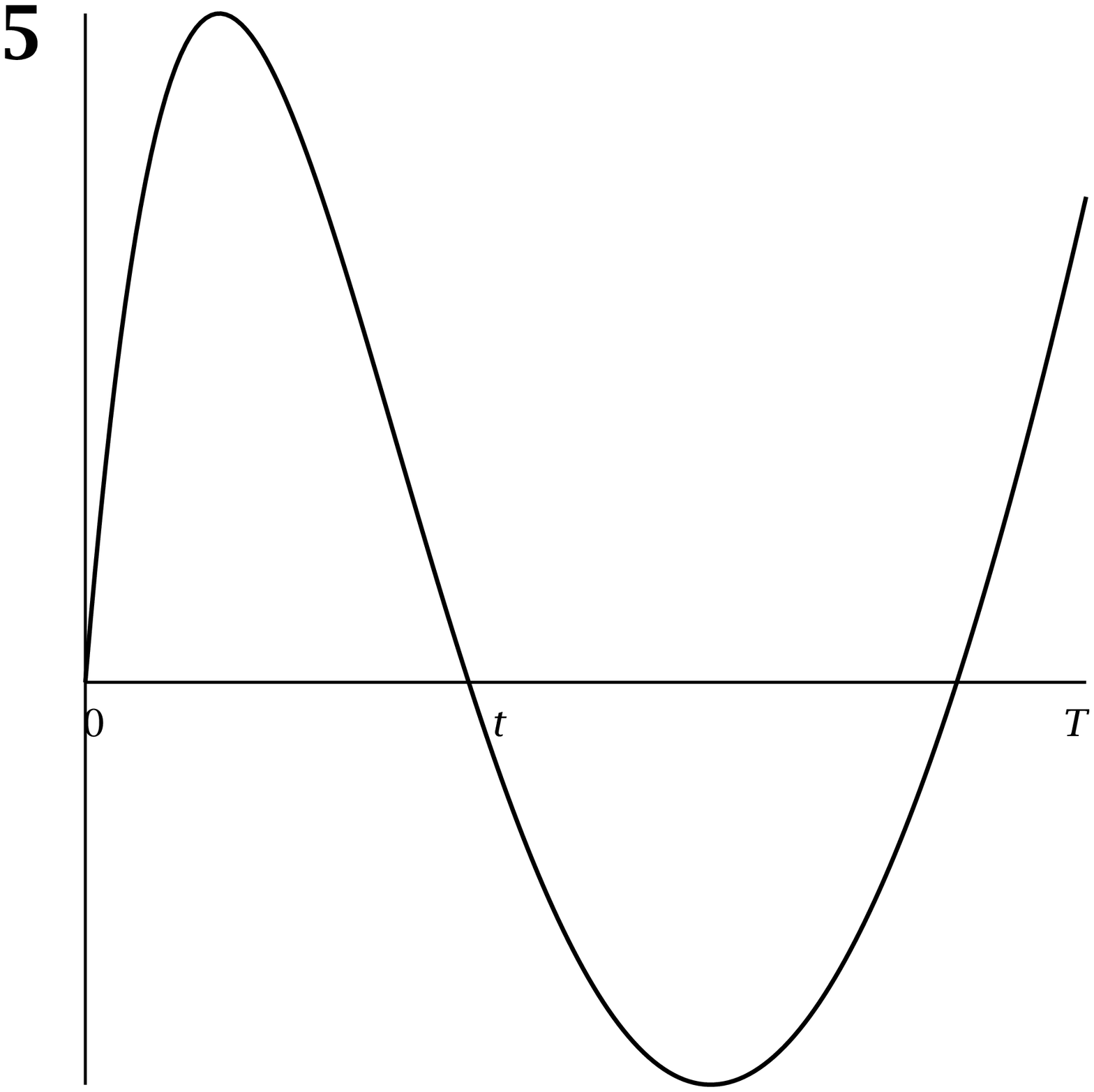}
\caption{Possible development scenarios for $f(t)$. \label{fig.1}}
\end{figure}
\noindent
Consider at first the derivative of $f(t)$.
\[
f'(t)=-b-\big(r-b+\frac{\sigma^2}a\big)e^{- at}+\frac{\sigma^2}a e^{-2 at}\;.
\]
Thus, in order to determine the behaviour of $f(t)$, substitute $e^{-at}$ by $t$ and consider the quadratic function $g(t):=-b-(r-b+\frac{\sigma^2}a)t+\frac{\sigma^2}a t^2$.
It is clear that $g(t)$ is a parabola opened upwards.
In particular, $g(t)$ has at most $2$ zeros $u_1$ and $u_2$:
 \begin{align}
 &D:=(r+b-\frac{\sigma^2}a)^2+4\frac{\sigma^2}ab\nonumber
 \\&u_1:=\frac{(r-b+\frac{\sigma^2}a)-\sqrt{D}}{\frac{2\sigma^2}a}\label{du}
 \\&u_2:= \frac{(r-b+\frac{2\sigma^2}a)+\sqrt{D}}{\frac{2\sigma^2}a}\;.\nonumber
 \end{align}
\begin{itemize}
\item If $D\le 0$, then $f(t)$ is increasing on $[0,T]$. 
\item If $D > 0$, then we have to consider $u_1$ and $u_2$ with $u_1<u_2$. 
\end{itemize}
Assume, $D>0$. The following $5$ scenarios are possible 
\begin{enumerate}
\item \label{1.2} $u_1\le e^{-a T}$ and $u_2\ge 1$. Then, $f(t)$ is decreasing on $[0,T]$.
\item \label{1.1} $u_2\le e^{-a T}$ or $u_1>1$. In this case $f(t)$ is increasing on $[0,T]$.
\item \label{1.3} $u_1\in (e^{-a T},1)$ and $u_2\ge 1$. Then, $f(t)$ is decreasing on $[0,w_2)$ and increasing on $(w_2,T]$, where
\begin{align}
w_2:=-\frac{\ln(u_1)}a\;.\label{r2}
\end{align}
\item \label{1.4} $u_1\le e^{-a T}$ and $u_2\in(e^{-aT},1)$. Then, $f(t)$ is increasing on $[0,w_1)$ and decreasing on $(w_1,T]$, where
\begin{align}
w_1:=-\frac{\ln(u_2)}a\;. \label{r1}
\end{align}
\item \label{1.5} $u_1,u_2\in (e^{-a T},1)$. Then, $f(t)$ is increasing on $[0,w_1)\cup (w_2,T]$ and decreasing on $(w_1,w_2)$.
\end{enumerate}
The possible development scenarios of $f(t)$ are illustrated in Figure \ref{fig.1}. 
\begin{Bem}
In particular, $f(t)$ is injective on $(-\infty,w_1)$, $[w_1,w_2]$ and on $(w_2,\infty)$, so that we can define inverse functions of $f$ acting just on the one of the above intervals:
\bigskip
\\
\begin{tabular}{lll}
$h_1:[f(w_1),1]\to (-\infty,w_1)$& &$f(t)\mapsto t$,
\\$h_2:[f(w_1),f(w_2)]\to [w_1,w_2]$& &$f(t)\mapsto t$,
\\$h_3:[f(T),f(w_2)]\to (w_2,\infty)$& &$f(t)\mapsto t$.
\end{tabular}
\bigskip
\\
In the case \ref{1.3}, we use just the functions $h_2$ on $[f(0),f(w_2)]$ and $h_3$ on $[f(T),f(w_2)]$. Considering \ref{1.4}, we define just $h_1$ on $[f(w_1),f(0)]$ and $h_2$ on $[f(w_1),f(T)]$.
\end{Bem}
For the sake of simplicity, we introduce
\begin{align}
&t_1:=h_1(f(T)),\quad \mbox{for the cases \ref{1.4} and \ref{1.5} given $f(T)\le f(0)$;}\label{t1}
\\&t_2:=h_2(f(T)),\quad \mbox{for the cases \ref{1.3} and \ref{1.5} given $f(T)\ge f(0)$}\label{t2}
\\&\hspace{3cm}\mbox{or $f(T)\ge f(w_1)$ correspondingly.}\nonumber
\end{align}
At first, we will consider the case where the payouts are bounded by some positive constant $\xi$, in the last part we consider the unrestricted case.
\section{The Optimal Strategy and the Value Function for the Zero-Bond Discounting \label{sec1}}
\noindent
We will consider just the fifth case, where $f$ has a maximum and a minimum. The other cases described above can be handled in a similar way. 
\subsection{$\xi\le \mu$}
\noindent
Since $\xi\le \mu$, the process remains non-negative even if we pay out on the maximal rate up to $T$. Thus, for a given pair $(t,x)\in[0,T]\times \R_+$ we have to compare $e^{f(t)}$ and $e^{f(T)}$. The optimal strategy $C^*= \{c^*_s\}$ is then given by 
\begin{equation}
c^*_s=\begin{cases}
\xi, & e^{f(t)}\ge e^{f(T)} 
\\0, & e^{f(t)}< e^{f(T)}
\end{cases}.\label{xikleinermu}
\end{equation}
The value function is then given by
\begin{equation}
V(t,x)=\int_t^T e^{f(s)} c^*_s\md s+(x+\int_t^T \mu-c^*_s \md s)e^{f(T)}\;. \label{value1}
\end{equation}
In particular, it holds $V_x(t,x)=e^{f(T)}$.
It is easy to check, that the value function solves the corresponding HJB equation \eqref{hjb},
is continuously differentiable with respect to $t$ and to $x$. 
Note, that in all five cases the optimal strategy does not depend on the initial capital $x$. 

%% file: Sec2.tex
\subsection{$\xi>\mu$\label{biger}}
\noindent
Here, the maximal payout boundary $\xi$ exceeds the drift $\mu$. 
Let $w_1$ and $w_2$ be the maximum and the minimum point of $f(t)$ correspondingly, defined in \eqref{r1} and \eqref{r2}. 
Note that if $f(w_1)\le f(T)$ it is optimal to wait until $T$ and pay out everything there. Obviously, the corresponding function will solve HJB Equation \eqref{hjb}.
\\Assume now $f(w_1)>f(T)$, i.e. $t_2$, see \eqref{t2}, is well-defined. 
We construct a candidate strategy $\tilde C=\{\tilde c_t\}$ applying a backward algorithm on the intervals $[t_2,T]$, $[t_1,w_1)$, $[w_1,t_2)$ and $[0,t_1)$, if $t_1$, \eqref{t1} exists; or on the intervals $[t_2,T]$, $[0,w_1)$, $[w_1,t_2)$ if $f(0)\ge f(T)$. W.l.o.g we assume $f(0)<f(T)$.
\\Let at first $t\in[t_2,T]$, then, $f(t)\le f(T)$ for all $t$. Let $\tilde c_t=0$ for $t \in [t_2,T]$, i.e. we wait until $T$ and pay out everything there. The corresponding return function $V_1(t,x):=\big(x+\mu (T-t)\big)e^{f(T)}$ obviously solves HJB Equation \eqref{hjb} on  $[t_2,T]\times\R_+$.
\\For $t\in [w_1,t_2)$ let
\[
\tilde c_t=
\begin{cases}
\xi, & x> 0
\\\mu, & x=0
\end{cases},
\]
yielding the return function 
\begin{align*}
&V_2(t,x)=\begin{cases}
\xi\int_t^{t_2}e^{f(s)}\md s+V_1\big(t_2,x+(\mu-\xi) (t_2-t)\big), & \frac{x}{\xi-\mu}+t\ge t_2
\\\xi\int_t^{\frac{x}{\xi-\mu}+t}e^{f(s)}\md s+\mu\int_{\frac{x}{\xi-\mu}+t}^{t_2}e^{f(s)}\md s+V_1(t_2,0), & \frac{x}{\xi-\mu}+t<t_2
\end{cases}
\\&
\frac{\md }{\md x}V_2(t,x)=\begin{cases}
e^{f(T)}, & \frac{x}{\xi-\mu}+t\ge t_2
\\e^{f\big(t+\frac{x}{\xi-\mu}\big)}, & \frac{x}{\xi-\mu}+t<t_2
\end{cases},
\end{align*}
which shows that $V_2$ solves HJB Equation \eqref{hjb}.
Consider now $t\in[t_1,w_1)$. The strategy will depend on the value of $\frac{\md }{\md x}V_2(w_1,x)$. 
Define on $[t_1,w_1)\times\R_+$
\[
\chi(t,x):=\inf\Big\{u>0:\; f(t+u)>f\Big(t+u+\frac{x+\mu u}{\xi-\mu}\Big)\Big\}.
\]
Note that the function $\chi(t,x)$ is a well-defined, continuously differentiable with respect to $x$ and to $t$ function. It holds $t+\chi(t,x)\le w_1$ and $f(t+\chi(t,x))=f\Big(t+\chi(t,x)+\frac{x+\mu \chi(t,x)}{\xi-\mu}\Big)$. 
For $t\in[t_1,w_1)$ let
\[
\tilde c_t=\begin{cases}
\xi, & \chi(t,x)=0
\\0, & \chi(t,x)>0
\end{cases}
\]
and the corresponding return function fulfils
\begin{align*}
&V_3(t,x)= \xi\int_{t+\chi(t,x)}^{t_2}e^{f(s)}\md s+V_2\Big(w_1,x+\chi(t,x)\xi+(\mu-\xi) (w_1-t)\Big)
\\&\frac{\md }{\md x}V_3(t,x)=\begin{cases}
e^{f(T)}, & \frac{x+\chi(t,x)\xi}{\xi-\mu}+t\ge t_2
\\e^{f\Big(t+\frac{x+\chi(t,x)\xi}{\xi-\mu}\Big)}, & \frac{x+\chi(t,x)\xi}{\xi-\mu}+t<t_2
\end{cases}.
\end{align*}
Hence, for the crucial condition in the HJB equation it holds due to the definition of $\chi(t,x)$:
\[
e^{f(t)}-\frac{\md }{\md x}V_3(t,x)=\begin{cases}
e^{f(t)}-e^{f(T)}> 0, & \frac{x+\chi(t,x)\xi}{\xi-\mu}+t> t_2
\\e^{f(t)}-e^{f(t+\chi(t,x))}\le 0, & \frac{x+\chi(t,x)\xi}{\xi-\mu}+t\le t_2
\end{cases},
\]
showing that $V_3$ solves the HJB equation on $(t_1,w_1)\times \R_+$.
\\It remains to consider $[0,t_1]$. There, for every $t$ it holds $f(t)<f(T)$. Let $\tilde c_t=0$ and the corresponding return function on $[0,t_1]\times \R_+$:
\[
V_{4}(t,x)=V_3(t_1,x+\mu(t_1-t))\;.
\]
It is easy to see that the function
\begin{equation}
V(t,x):=\begin{cases}
V_1(t,x), & t\in[t_2,T]
\\V_2(t,x), & t\in[w_1,t_2)
\\V_3(t,x), & t\in[t_1,w_1)
\\V_4(t,x), & t\in[0,t_1)
\end{cases}
\label{value2}
\end{equation}
is continuously differentiable with respect to $x$ and to $t$.
\begin{Satz}
If $\xi\le \mu$, the optimal strategy and the value function are given in \eqref{xikleinermu} and in \eqref{value1} respectively.
If $\xi>\mu$, the optimal strategy is $\tilde C$, described in Subsection \ref{biger}, and the value function is given in \eqref{value2}.
\end{Satz}
\begin{proof}
Since the proof methods are well-known, we just refer to, for example, Fleming and Soner \cite{fleming}.
\end{proof}
Next, we will consider the case with unrestricted payments, i.e. $\xi\to\infty$.
\subsection*{Unrestricted Payments}
\noindent
The case of unrestricted payments is very easy. Basically, one has to wait until a local maximum and pay out the available capital there.
The corresponding HJB equation is
\[
V_t+\mu V_x+ \sup\limits_{c\ge 0}c\{e^{f(t)}-V_x\}=0\;.
\]
Considering again the fifth case ($f$ has a maximum and a minimum), we have to distinguish between $f(w_1)\ge f(T)$ and $f(w_1)< f(T)$.
\\If $f(w_1)\le f(T)$ then for all $t\in[0,T]$ it is optimal to wait until $T$ and pay out everything there, yielding as the value function $\big(x+\mu(T-t)\big)e^{f(T)}$.
\\Assume now $f(w_1)> f(T)$. For $t\in[t_2,T]$, it is optimal to wait until $T$ and pay out everything there:
\[
V_{1}(t,x)= \big(x+\mu(T-t)\big)e^{f(T)}\;.
\]
For $t\in [w_1,t_2)$, pay out the initial capital immediately, pay on the rate $\mu$ until $t_2$, wait then until $T$ and pay out the collected drift there:
\[
V_{2}(t,x)=xe^{f(t)}+\mu \int_{t}^{t_2}e^{f(s)}\md s+V_{1}(t_2,0)\;.
\]
And finally, for $t\in[0,w_1]$ we have to distinguish between $f(0)\ge f(T)$ and $f(0)< f(T)$. W.l.o.g. we let $f(0)< f(T)$, i.e. $t_1$ exists.
For all $t\in[t_1,w_1)$, one has to wait until the maximum $w_1$:
\[
V_{3}(t,x)=\big(x+\mu(w_1-t)\big)e^{f(w_1)}+V_2(w_1,0)\;.
\]
For $t\in[0, t_1)$ just wait until $t_1$.
\[
V_{4}(t,x)=V_{3}(t_1,x+\mu(t_1-t))\;.
\]
Since the proof methods are well-known, we omit further explanations and just refer to, for example, Schmidli \cite[p. 102]{HS}. 
\\Note that the backward algorithms for both, restricted and unrestricted payments, can be applied for an arbitrary continuously differentiable interest rate function, like for example sine or cosine.

%% file: Bsp.tex
\begin{Bsp}
\begin{figure}
\includegraphics[scale= 0.5, bb = -170 220 200 500]{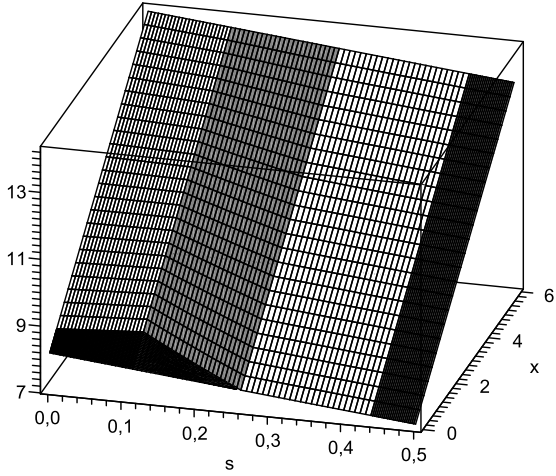}
\caption{The value function $V(s,x)$. \label{fig:bsp}}
\end{figure}
Let $r_0=-0.2$, $b= -0.1$, $a= 1$, $\sigma= 1$, $\mu = 2$, $\xi = 4$ and $T=4$.
Thus, $w_1=0.2611$ and $w_2=2.0414$, $t_1=0.1134$ and $t_2=0.4388$.
Note that it holds $f(w_1)>f(4)>f(0)$.
\\
For $(s,x)\in[t_2,T]\times\R_+$, $f$ is increasing in $s$. We wait until $T$ and pay out everything there. The value function for this area is given by the right (black) slice in Figure \ref{fig:bsp}.
\\In $[w_1,t_2)$ we pay on the maximal possible rate up to $t_2$, white slice (the second from the right) in the picture. 
\\For $s\in[t_1,w_1)$, we either wait until $t+\chi(t,x)$ or start immediately paying on the rate $\xi$ up to $w_1$: second slice from the left in Figure \ref{fig:bsp}. In the black area we wait, in the gray area we pay.  
The value function for $s\in[0,t_1)$ is given by the left slice. Like for $t\in[t_1,w_1)$ we wait in the black area and pay in the white area.
\end{Bsp}
\begin{Bem}[Arbitrary drift function]
Consider the process 
\begin{align*}
X_t= x+ \int_0^t\mu(s)\md s\;,
\end{align*}
where $\mu(s)$ is an arbitrary continuous function with finitely many zeros in $[0,T]$.
An admissible strategy $C$ denotes now the cumulated consumption, is c\`{a}dl\`{a}g, increasing and $\Delta C_s\leq X_{s-}^C$.
The HJB equation in this case is
\[
\max\{V_t+\mu(t) V_x,e^{f(t)}-V_x\}=0\;.
\] 
The problem of consumption maximization for unrestricted payments with deterministic constant interest rate $\delta>0$ was considered in \cite{jgt}. There, it was possible to establish an algorithm for finding an explicit expression for the optimal strategy and the value function. Here, the algorithm for a constant interest rate from \cite{jgt} has to be combined with the algorithm for a constant drift with pure-discount bond described earlier in this paper. 
However, the finding procedure of the value function would be very time- and spaceconsuming.  
\\An interested reader can contact the author for further information.
\end{Bem}

%% file: sec5.tex
\section{Stochastic Interest Rates}
\noindent
In this section, we consider a model with a stochastic discounting rate and an infinite time horizon. Like before, we assume that the surplus of the considered household is
\[
X_t=x+\mu t\;.
\]
\subsection{Geometric Brownian Motion as a Discounting Process}
\noindent
In this subsection, we let $r_t=r+mt+\sigma W_t$, where $\{W_t\}$ is a standard Brownian motion.
Our target is to maximize the expected discounted consumption over all admissible strategies $C=\{c_s\}$, if the discounting process is given by a geometric Brownian motion.
It means, we assume that the consumption behaviour of the considered household is linked to a stock price modelled by a geometric Brownian motion.  
\\As an admissible strategy we denote all $C=\{c_s\}$ such that $c_s\in[0,\xi]$, $C$ is adapted to the filtration $\{\mathcal F_s\}$, generated by $\{W_s\}$ and $X_t^C=X_t-\int_0^t c_s\md s \ge 0$ for all $t\ge 0$ (i.e. consumption cannot cause ruin). 
The return function corresponding to a strategy $C=\{c_s\}$ and the value function are defined as
\begin{align*}
&V^C(r,x)=\mE\Big[\int_0^\infty e^{-r_s} c_s \md s|r_0=r\Big]\;,\quad (r,x)\in\R\times \R_+\;,
\\& V(t,x)=\sup\limits_C V^C(r,x)\quad (r,x)\in\R\times \R_+\;.
\end{align*}
Note that $\mE[e^{r_u}]=e^{-r-(m-\frac{\sigma^2}2)u}$. In order to guarantee the well-definiteness of the value function, we assume $m>\frac{\sigma^2}2$.
Obviously,
\[
V(r,x)\le \xi \mE\Big[\int_0^\infty e^{-r-(m-\frac{\sigma^2}2)t}\md t\Big]\;,
\]
The above integral is finite for all $r\in\R$. 
The HJB equation corresponding to the problem is
\begin{align}
\mu V_x+m V_r+\frac{\sigma^2}2 V_{rr}+\sup\limits_{0\le c\le \xi}c\Big\{e^{-r}-V_x\Big\}=0\;.\label{hjb2}
\end{align}
Consider at first the case when the boundary $\xi$ is smaller or equal to the drift $\mu$.
Here, we can just pay out on the maximal rate $\xi$ up to $\infty$ without ruining. The return function $V^\xi$ corresponding to such a strategy is then given by
\[
V^\xi(r,x)=\xi\mE\Big[\int_0^\infty e^{-r_s}\md s \Big]=\xi \int_0^\infty e^{-r -\big(m-\frac {\sigma^2}2\big)s} \md s=\frac{\xi e^{-r}}{m-\frac{\sigma^2}2}\;.
\]
$V^\xi$ does not depend on $x$ and  obviously solves HJB Equation \eqref{hjb2}.
\\Consider now $\xi>\mu$. Now it is impossible to pay out on the rate $\xi$ till the end of the time.
Instead, we consider the strategy $\hat C=\{\hat c_s\}$
\begin{align}
\hat c_s=\begin{cases}
\xi & 0 \le s\le \frac{x}{\xi-\mu}
\\\mu & s> \frac{x}{\xi-\mu}\label{strat1}
\end{cases}\;.
\end{align}
The corresponding return function is given by
\[
V^{\hat C}(r,x)=\xi \int_0^{\frac x{\xi-\mu}} e^{-r-(m-\frac{\sigma^2}2)s}\md s+\mu \int_{\frac x{\xi-\mu}}^\infty e^{-r-(m-\frac{\sigma^2}2)s}\md s\;.
\] 
\begin{Satz}
The strategy $\hat C$, defined in \eqref{strat1}, is the optimal strategy and $V^{\hat C}(r,x)$ is the value function.
\end{Satz}
\begin{proof}
Consider the function $V^{\hat C}(r,x)$. It holds
\[
V_x^{\hat C}(r,x)=e^{-r-(m-\frac{\sigma^2}2)\frac x{\xi-\mu}}\;.
\]
Thus, for all $x\ge 0$ it holds
\[
e^{-r}-V_x^{\hat C}(r,x)=e^{-r}\Big(1-e^{-(m-\frac{\sigma^2}2)\frac x{\xi-\mu}}\Big)\ge 0\;.
\]
It is easy to see that the function $V^{\hat C}$ solves HJB equation \eqref{hjb2}.
\\It remains to prove that $V^{\hat C}(t,x)=V(t,x)$. Let $C=\{c_s\}$ be an arbitrary admissible strategy, then holds
\begin{align*}
V^{\hat C}(r_t,X_t^C)&=V^{\hat C}(r,x)+\int_0^t (\mu-c_s) V^{\hat C}_x(r_s,X^C_s)+mV^{\hat C}_r(r_s,X^C_s)+\frac{\sigma^2}2 V^{\hat C}_{rr}(r_s,X_s^C)\md s
\\&\quad {}+ \sigma \int_0^t V^{\hat C}_r(r_s,X_s^C) \md W_s
\\&\le  -\int_0^t e^{-r_s} c_s\md s+\sigma \int_0^t V^{\hat C}_r(r_s,X_s^C) \md W_s\;.
\end{align*}
Because $V^{\hat C}$ is bounded, the stochastic integral above is a martingale with expectation zero. Further,
\[
\mE\big[V^{\hat C}(r_t,X_t^C)\big]\le \mE\big[e^{-r-mt-\sigma W_t}\big]=e^{-r}e^{-(m-\frac{\sigma^2}2)t}\;.
\]
Thus, applying the expectations and letting $t\to \infty$ yields $$\mE\Big[\int_0^t e^{-r_s} c_s\md s\Big] \le V^{\hat C}(r,x)$$.
\end{proof}
For unrestricted payments the HJB equation is
\[
\max\{\mu V_x+m V_r+\frac{\sigma^2}2 V_{rr},e^{-r}-V_x\}=0\;
\]
And, it is easy to see that the value function is given by
\[
V(r,x)=e^{-r}x+e^{-r}\mu\int_0^\infty  \mE\Big[e^{-mt-\sigma W_t}\Big]\md t=e^{-r}x+e^{-r}\frac{\mu}{m-\frac{\sigma^2}2}\;.
\]
It means, we have to pay out the initial capital immediately and to pay on the rate $\mu$ up to the infinite time horizon.
For the proof methods confer for example Schmidli \cite[p. 102]{HS}.

%% file: Sec4.tex
\subsection{Ornstein-Uhlenbeck Process a Short Rate}
\noindent
Like in Section \ref{sec:1}, we denote again by $\{r_s\}$ an Ornstein-Uhlenbeck process
\[
r_s=re^{-a s}+\tilde b(1-e^{-a s})+\tilde\sigma e^{-as}\int_0^se^{au}\md W_u\;,
\]
where $\{W_u\}$ is a standard Brownian motion, $a,\tilde\sigma>0$, and let $U^r_s=\int_0^s r_u\md u$ with $r_0=r$.
Our target is to maximize the expected discounted consumption over all admissible strategies $C=\{c_s\}$, if the interest rate is given by $\{r_t\}$. A strategy $C=\{c_s\}$ is called admissible if $c_s\in[0,\xi]$, is adapted to the filtration $\{\mathcal F_s\}$, generated by $\{r_s\}$ and $X^C_t=X_t-\int_0^t c_s\md s\ge 0$ for all $t\ge 0$.
\\
Here, we assume that the long-term mean $\tilde b$ of the process $\{r_s\}$ fulfils: $\tilde b> \frac{\tilde \sigma^2}{2a^2}$. 
The return function corresponding to a strategy $C=\{c_s\}$ and the value function are defined by
\begin{align*}
&V^C(r,x)=\mE\Big[\int_0^\infty e^{-U^r_s} c_s \md s|X_0=x\Big],\quad (r,x)\in\R\times\R_+\;,
\\& V(r,x)=\sup\limits_C V^C(r,x),\quad (r,x)\in\R\times\R_+\;.
\end{align*}
Since $r$ is now a variable and not a constant parameter like in Section \ref{sec1}, we manifest this fact by writing $f(r,s)$ instead of $f(s)$ for the function $f$ defined in \eqref{f}. 
Denoting again $\sigma:=\frac{\tilde \sigma}{\sqrt{2a}}$ and $b:=\tilde b-\frac{\tilde \sigma^2}{2a^2}>0$, we have $\mE[e^{-U^r_s}]=e^{f(r,t)}$.
The HJB equation corresponding to the problem is
\begin{align}
\mu V_x+a(\tilde b-r) V_r+\frac{\tilde \sigma^2}2 V_{rr}-r V+\sup\limits_{0\le c\le \xi}c\Big\{1-V_x\Big\}=0\;.\label{hjb3}
\end{align} 
Further, the function $e^{f(r,s)}$ can be estimated as follows
\begin{align*}
e^{f(r,t)}&=\exp\Big\{-bt-\frac{r-b}a(1-e^{-at})-\frac{\sigma^2}{2a^2}(1-e^{-at})^2\Big\}
\\&\le \exp\Big\{-bt-\min\Big(\frac{r-b}a,0\Big)\Big\}.
\end{align*}
Using the above estimation and the fact $b>0$, we find the following boundary for the value function:
\begin{equation}\label{schranke}
\begin{split}
V(r,x)&\le \xi \mE\Big[\int_0^{\infty} e^{-U^r_s}\md s\Big]
=\xi\int_0^\infty e^{f(r,s)}\md s\le \frac \xi b\exp\Big\{-\min\Big(\frac{r-b}a,0\Big)\Big\}\,,
\\V(r,x)&\ge \xi \int_0^{\frac{x}{\xi-\mu}\vee 0} e^{f(r,s)}\md s+ \mu \int_{\frac{x}{\xi-\mu}\vee 0}^\infty e^{f(r,s)}\md s\,.
\end{split}
\end{equation} 
for every choice of $a,\sigma>0$ and all $(r,x)\in\R\times\R_+$.
\subsubsection{Restricted rates with $\xi\le \mu$.}
\noindent
Assume first $\xi\le \mu$. In this case the process $X^\xi_t=x+(\mu-\xi)t$ will never hit zero. The return function $V^\xi$ corresponding to the constant strategy $c_s\equiv \xi$ is given by:
\begin{align*}
V^\xi(r,x)&= \xi\mE\Big[\int_0^\infty e^{-U^r_s}\md s\Big]
=\xi\int_0^\infty e^{f(r,s)}\md s\;.
\end{align*}
Note that $V^\xi$ does not depend on $x$ in this case. In particular:
\[
1-V_x^\xi(r,x)=1\;.
\]
It is an easy exercise to prove that $V^\xi$ solves the ODE
\[
a(\tilde b-r) v_r+\frac{\tilde \sigma^2}2 v_{rr}-r v+\xi=0\;.
\]
For $V^\xi(r,x)$ it is possible to interchange integration and differentiation so that
\begin{align*}
&V^\xi_r(r,x)=\xi\int_0^\infty -\frac{1-e^{-as}}a e^{f(r,s)}\md s,
\\&V^\xi_{rr}(r,x)=\xi\int_0^\infty \frac{(1-e^{-as})^2}{a^2} e^{f(r,s)}\md s\;.
\end{align*}
Thus,
\[
a(\tilde b-r) V^\xi_r(r,x)+\frac{\tilde \sigma^2}2 V^\xi_{rr}(r,x)-r V^\xi(r,x)=\xi\int_0^\infty f_s(r,s)e^{f(r,s)}\md s=-\xi e^{f(r,0)}=-\xi\;,
\]
which proves our claim.
Here, the function $V^\xi$ becomes a candidate for the value function.
\subsubsection{Restricted rates with $\xi> \mu$.}
\noindent
Assume now $\xi>\mu$.
The return function corresponding to the strategy 
\begin{align}
\hat c_s=\begin{cases}
\xi & 0 \le s\le \frac{x}{\xi-\mu}
\\\mu & s> \frac{x}{\xi-\mu}\label{strat}
\end{cases}
\end{align}
is given by
\begin{align*}
V^{\hat C}(r,x)=\mE\bigg[\xi\int_0^{\frac x{\xi-\mu}} e^{-U^r_s} \md s+\mu\int_{\frac x{\xi-\mu}}^\infty e^{-U^r_s}\md s\bigg]
=\xi \int_0^{\frac x{\xi-\mu}} e^{f(r,s)}\md s+\mu\int_{\frac x{\xi-\mu}}^\infty e^{f(r,s)}\md s.
\end{align*}
Obviously, $V^{\hat C}$ is continuously differentiable with respect to $x$ and twice continuously differentiable with respect to $r$.
Like in the case $\xi\le\mu$, we can interchange integration and derivation and obtain
\begin{align*}
a(\tilde b-r) V^{\hat C}_r+\frac{\tilde \sigma^2}2 V^{\hat C}_{rr}-r V^{\hat C}&=\xi\int_0^{\frac x{\xi-\mu}}f_s(r,s)e^{f(r,s)}\md s+\mu\int_{\frac x{\xi-\mu}}^\infty f_s(r,s)e^{f(r,s)}\md s
\\&=(\xi-\mu) e^{f\big(r,\frac x{\xi-\mu}\big)}-\xi\;.
\end{align*}
The derivative of $V^{\hat C}$ with respect to $x$ is given by
\[
V^{\hat C}_x(r,x)=e^{f\big(r,\frac x{\xi-\mu}\big)}\;.
\]
And we can conclude that $V^{\hat C}$ solves the PDE
\begin{equation*}
(\mu-\xi) v_x+a(\tilde b-r) v_r+\frac{\tilde \sigma^2}2 v_{rr}-r v+\xi=0\;.
\end{equation*}
Note that $1-V_x^{\hat C}\ge 0$ iff $f\big(r,\frac x{\xi-\mu}\big)\le 0$.
In order to find out whether $V^{\hat C}$ could become a good candidate for the value function, we have to investigate the properties of the function $f(r,s)$.
\smallskip
\\Due to Subsection \ref{properties}, for a fixed $r$ and $b>0$ the function $f_s(r,s)$ can have at most one zero at $s=w_1(r)=-\frac 1a\ln(u_1(r))$ with
\[
u_1(r)=\frac{r-b+\frac{\sigma^2}a+\sqrt{\big(r-b+\frac{\sigma^2}a\big)^2+4b\frac{\sigma^2}a}}{2\sigma^2/a}>0\;.
\]
Note that $f_{ss}(r,s)=-af_s(r,s)-a\big(b+\frac{\sigma^2}ae^{-2as}\big)$. This means that for a fixed $r$ it holds either $f_s(r,s)\le 0$ on $[0,\infty)$, if $u_1(r)\ge 1$, or $f_s(r,s)>0$ on $[0,w_1(r))$ and $f_s(r,s)<0$ on $(w_1(r),\infty)$, if $u_1(r)<1$.
Consequently, we consider just the cases \ref{1.2} and \ref{1.4} in Subsection \ref{properties}, illustrated in Pictures $1$ and $4$ in Figure \ref{fig.1}. 
It is easy to see that the function $u_1(r)$ is increasing in $r$ and $u_1(0)=1$. It means that $f(r,s)<0$ for all $(r,s)\in\R_+^2$. Thus, for the strategy $\hat C$ defined in \eqref{strat} it holds
\[
V_x^{\hat C}(r,x)=e^{f\big(r,\frac{x}{\xi-\mu}\big)}\le 1\;\quad (r,x)\in\R_+^2\;.
\]
If $r<0$ and $s>0$, then for every fixed $r\in\R_-$ the function $f(r,s)$ attains its maximum at $w_1(r)$. Further, since $f(r,0)=0$ for all $r\in \R$ and $\lim\limits_{s\to\infty}f(r,s)=-\infty$ the curve
\begin{equation*}
\alpha(s):=\frac a{1-e^{-as}}\Big\{-bs+\frac ba(1-e^{-as})-\frac{\sigma^2}{2a^2}(1-e^{-as})^2\Big\}
\end{equation*}
is unique with $f\big(\alpha(s),s\big)\equiv 0$. 
Using the power series representation of the logarithm function, confer for example \cite[p. 381]{hh}, it holds for $s>0$:
\begin{align*}
\alpha(s)&=\frac a{1-e^{-as}}\Big\{-\frac ba\s_{n=1}^\infty \frac{(1-e^{-as})^n}{n}+\frac ba(1-e^{-as})-\frac{\sigma^2}{2a^2}(1-e^{-as})^2\Big\}
\\&= 
-b\s_{n=1}^\infty \frac{(1-e^{-as})^n}{n+1}-\frac{\sigma^2}{2a}(1-e^{-as})<0,
\\
\alpha'(s)&= -ba\cdot e^{-as}\s_{n=1}^\infty \frac{(1-e^{-as})^{n-1}n}{n+1}-\frac{\sigma^2}{2}e^{-as}<0\;.
\end{align*}
Thus, $\alpha$ is negative and strictly decreasing. Let $\beta(r)$ denote the inverse function of $\alpha(s)$ for $r\in(-\infty,0)$ (is well-defined because $\alpha$ is strictly decreasing), i.e. $\beta(\alpha(s))=s$. Then $\beta(r)$, $r\in\R_-$, is positive and strictly decreasing. In particular, $f(r,s)>0$ for $s<\beta(r)$ and $f(r,s)<0$ for $s>\beta(r)$ and $V_{xx}^{\hat C}(r,x)<0$ for $x\ge \beta(r)$.
Thus, the function $V^{\hat C}$ could not be the value function. 
\bigskip
\begin{Satz}\label{properties2}
The value function $V(r,x)$ is locally Lipschitz continuous, strictly increasing and concave in $x$; locally Lipschitz continuous, decreasing and convex in $r$. It holds $\lim\limits_{r\to\infty} V(r,x)=0$.
\end{Satz}
\begin{proof}
$\bullet$  Let at first $h>0$, $r\in \R$ and $C$ be an admissible $\varepsilon$-optimal strategy for $(r+h,x)$. Then, $C$ is also an admissible strategy for $(r,x)$ (the argument works also the other way round) and it holds
\begin{align*}
V(r+h,x)-V(r,x)&\le V^C(r+h,x)+\varepsilon-V^C(r,x)
\\&=\mE\Big[\int_0^\infty e^{-U^r_s}c_s\big(e^{-\frac ha (1-e^{-as})}-1\big)\md s\Big]+\varepsilon\le 0\;.
\end{align*}
Considering an $\varepsilon$ optimal strategy for $(r,x)$ and applying the same arguments yields
\[
V(r+h,x)-V(r,x)\ge- V(r,x) \frac ha\ge -\frac {h\xi}{ba} \exp\Big(-\min\Big(\frac{r-b}a,0\Big)\Big)\;.
\]
Thus, $V$ is locally Lipschitz continuous and in particular continuous in $r$. 
\medskip
\\$\bullet$ For $r,q\in\R$, $\lambda\in(0,1)$ let $z=\lambda r+(1-\lambda)q$ and $\tilde C$ be an $\varepsilon$-optimal strategy for $(z,x)$. Then,
\begin{align*}
V(z,x)-\varepsilon\le V^{\tilde C}(z,x)&=\int_0^\infty e^{-U^z_s} \tilde c_s\md s= \int_0^\infty e^{-\lambda U^r_s-(1-\lambda)U^q_s} \tilde c_s\md s
\\&\le \lambda \int_0^\infty e^{-U^r_s} \tilde c_s\md s+ (1-\lambda)\int_0^\infty e^{-U^q_s} \tilde c_s\md s\;.
\end{align*}
Note that $\tilde C$ is an admissible strategy for $(r,x)$ as well as for $(q,x)$. Thus,
\[
V(z,x)\le \lambda V(r,x)+(1-\lambda) V(q,x)\;,
\]
i.e. $V$ is convex in $r$. 
\medskip
\\$\bullet$ For every $h>0$, it is clear that an admissible strategy for $(r,x)\in\R\times \R_+$ is also admissible for $(r,x+h)$, which implies that $V$ is increasing in the $x$ component.
\\On the other hand, let $C$ be an $\varepsilon$-optimal strategy for the starting point $(r,x+h)$ and define $\tilde C=\{\tilde c_s\}$ to be
\[
\tilde c_s=\begin{cases}
0 & s<\frac h\mu\\
c_{s-\frac h\mu} & s\ge \frac h\mu
\end{cases}\;.
\]
Obviously, $\tilde C$ is an admissible strategy for the starting point $(r,x)$. Then, we obtain
\begin{align*}
V(r,x+h)-V(r,x)&\le V^C(r,x+h)+\varepsilon - V^{\tilde C}(r,x)
\\&=\mE\Big[\int_0^\infty e^{-U_s^r}c_s\md s\Big]-\mE\Big[\int_{h/\mu}^\infty e^{-U_s^r}c_{s-h/\mu}\md s\Big]+\varepsilon
\\&=\mE\Big[\int_0^\infty e^{-U_s^r}c_s \big\{1-e^{-\int_0^{h/\mu} r_{s+u}\md u}\big\}\md s\Big]+\varepsilon\;.
\end{align*}
Let $\tilde U_{h/\mu}^{r_s}:=\int_0^{h/\mu} r_{s+u}\md u$, and note that $\tilde U_{h/\mu}^{r_s}$ depends on $U_s^r$ just via $r_s$. Then noting that the random variable $r_s$ is normally distributed (with mean $re^{-as}+\tilde b(1-e^{-as})$ and variance $\frac{\tilde \sigma^2}{2a}(1-e^{-2as})$), using $1-e^{x}\le -x$ and the definition of $f$ in \eqref{f}, we obtain the following estimation 
\begin{align*}
\mE\Big[\int_0^\infty e^{-U_s^r}c_s &\big\{1-e^{-\tilde U_{h/\mu}^{ r_{s}}}\big\}\md s\Big]=\int_0^\infty \mE\Big[\mE\big[e^{-U_s^r}c_s \big(1-e^{-\tilde U_{h/\mu}^{r_s}}\big)|r_s\big]\Big]\md s
\\&= \int_0^\infty \mE\Big[\mE\big[e^{-U_s^r}c_s|r_s\big] \Big\{1-e^{f(r_s,\frac h\mu)}\Big\}\Big]\md s
\\&\le -\int_0^\infty \mE\Big[\mE\big[e^{-U_s^r}c_s|r_s\big] f\Big(r_s,\frac h\mu\Big) \Big]\md s
\\&= \int_0^\infty \mE\Big[e^{-U_s^r}c_s \Big\{b\frac h\mu+\frac{\sigma^2}{2a^2}(1-e^{-ah/\mu})^2+\frac{r_s-b}a(1-e^{-ah/\mu})\Big\}\Big]\md s 
\\&\le \Big(b+\frac{\sigma^2}{2a}\Big)\frac {h\xi}{b\mu} e^{-\min\big(\frac{r-b}a,0\big)}+\int_0^\infty \mE\Big[e^{-U_s^r}c_s \frac{r_s-b}a\big(1-e^{-ah/\mu}\big)\Big]\md s \;.
\end{align*}
Consider now the function $\Theta(r,s,y):=\mE\big[e^{-U_s^r}|r_s=y\big]$. Using Borodin and Salminen, \cite[p. 525]{bs}, one finds
\begin{align*}
\Theta(r,s,y)&=\exp\Big\{-\tilde bs-\frac{r+y-2\tilde b}a\,{\rm tanh}\Big(\frac{as}2\Big)+\frac{\sigma^2}{a^2}\Big(as-2\,{\rm tanh}\Big(\frac{as}2\Big)\Big)\Big\}
\\&= \exp\Big\{-bs-\frac{r- b}a\,{\rm tanh}\Big(\frac{as}2\Big)-\frac{y-  b}a\,{\rm tanh}\Big(\frac{as}2\Big)\Big\}
\\&\le \exp\Big\{-bs-\min\Big(\frac{r- b}a,0\Big)-\min\Big(\frac{y-b}a,0\Big)\Big\}\;.
\end{align*}
Thus, it holds
\begin{align*}
\int_0^\infty \mE\Big[e^{-U_s^r}c_s \frac{r_s-b}a\big(1-e^{-ah/\mu}\big)\Big]\md s&\le \frac {h\xi}\mu\int_0^\infty \mE\Big[\mE\big[e^{-U_s^r}|r_s\big]\cdot (r_s-b)\one_{[r_s>b]}\Big]\md s
\\& \le \frac {h\xi}\mu e^{-\min\big(\frac{r- b}a,0\big)}\int_{0}^\infty e^{-bs}\mE\big[(r_s-b)\one_{[r_s>b]}\big]\md s\;.
\end{align*}
Note that since $r_s$ is normally distributed, the expected value above can be estimated as follows
\begin{align*}
\mE\big[(r_s-b)\one_{[r_s>b]}\big]&=\frac{(r-\tilde b)e^{-as}+\tilde b-b}2\bigg(1+{\rm erf}\bigg(\frac{(r-\tilde b)e^{-as}+\tilde b-b}{\sigma\sqrt{2(1-e^{-2as})}}\bigg)\bigg)
\\&\quad{}+\frac{\sigma\sqrt{1-e^{-2as}}}{\sqrt{2\pi}}e^{-\frac{((r-\tilde b)e^{-as}+\tilde b-b)^2}{2(1-e^{-2as})\sigma^2}}
\\&\le \sigma+\begin{cases} 
(r-\tilde b)e^{-as}+\tilde b-b &\mbox{: for all $s\ge -\ln\big(\frac{\tilde b-b}{\tilde b-r}\big)/a$ and $r\le b$,} \\
(r-\tilde b)e^{-as}+\tilde b-b &\mbox{: for all $s\ge 0$ and $r> b$,}\\
0 &\mbox{: otherwise.}
\end{cases} 
\end{align*}
Thus, defining 
\[
\Lambda:= \frac{\sigma(a+b)}b+a\frac{\tilde b-b}{b}+\frac {(a+b)}b\big(b+\frac{\sigma^2}{2a}\big)
\]
we obtain
\[
V(r,x+h)-V(r,x)\le  \frac{h\xi}{\mu(a+b)}\Big(\max(r-b,0)+\Lambda\Big)e^{-\min\big(\frac{r- b}a,0\big)}\;.
\]
\medskip
\\$\bullet$ In order to prove the convexity in the $x$ component, let $x,y\ge0$, $C^x$ be an $\varepsilon$-optimal strategy for $(r,x)$ and $C^y$ be an $\varepsilon$-optimal strategy for $(r,y)$. Then, for $z=\lambda x+(1-\lambda)y$:
\[
0\le \lambda\big(x+\mu t-C^x_t\big)+(1-\lambda)\big(y+\mu t-C^y_t\big)=z+\mu t-\big(\lambda C^x_t+(1-\lambda)C^y_t\big)\;.
\]
Thus, $\lambda C^x+(1-\lambda)C^y$ is an admissible strategy for $(r,z)$. Since $\varepsilon$ was arbitrary, we can conclude
\[
\lambda V(r,x)+(1-\lambda) V(r,y)\le V(r,z)\;,
\]
i.e. $V$ is concave in $x$.
\\Further, we know that the value function is bounded, and using the monotone convergence theorem (since $f(r,s)$ is decreasing in $r$) we obtain
\[
\lim\limits_{r\to\infty} V(r,x)\le\lim\limits_{r\to\infty}\xi\int_0^\infty e^{f(r,s)}\md s=0\;.
\]
\medskip
\\$\bullet$ Estimation of the difference quotient of the value function with respect to $r$.
\\Define now an auxiliary function $\tilde V^C(r,x):= \mE\Big[\int_0^\infty e^{-U_s^r}c_s(1-e^{-as})\md s\Big]$ and let $C$ be an admissible strategy, $h>0$. Then
\begin{align*}
V^C(r+h,x)&=\mE\Big[\int_0^\infty e^{-U_s^{r+h}}c_s\md s\Big]=\mE\Big[\int_0^\infty e^{-U_s^{r}}c_s e^{-\frac ha(1-e^{-as})}\md s\Big]
 \\&\ge V^C(r,x)-\frac ha \tilde V^C(r,x)\,,
\\V^{C}(r,x)&=\mE\Big[\int_0^\infty e^{-U_s^{r}} c_s\md s\Big]
=\mE\Big[\int_0^\infty e^{-U_s^{r+h}} c_se^{\frac ha(1-e^{-as})}\md s\Big]
\\&\ge V^C(r+h,x)+\frac ha \tilde V^C(r+h,x)\;.
\end{align*}
Let $h>0$ and $C$ an $h^2$-optimal strategy for $(r,x)$, then
\begin{align*}
\frac {1}a \tilde V^{C}(r,x)+h\ge \frac{V^{C}(r,x)-V^{C}(r+{h},x)}{h}+h\ge \frac{V(r,x)-V(r+h,x)}{h}\;.
\end{align*}
Since, $V$ is convex in $r$ we obtain
\begin{align*}
&\frac{V(r,x)-V(r+h,x)}{h}\ge \frac{V(r-h,x)-V(r,x)}{h}
\\&\quad\ge \frac{V^{C}(r-h,x)-V^{C}(r,x)}{h}-h\ge \frac 1a \tilde V^{C}(r,x)-h\;.
\end{align*}
\end{proof}
It has been shown that the value function is convex in $r$ and concave in $x$. 
We conjecture that the optimal strategy is of a barrier type, i.e. we pay on the maximal rate above some barrier and do nothing below this barrier, whereas the barrier for $x$ should be equal to $0$ and the barrier for $r$ should be given by some constant $r^*$. Then, we have to consider two functions, describing the value function above and below the barrier. 
Unfortunately, we were not able to find a closed expression for a return function corresponding to such a barrier strategy.
That is why, we switch to the viscosity ansatz.
\begin{Def}\label{def:1}
\rm We say that a continuous function
$\underline{u}:\R\times\R_+\to \R_+$ is a viscosity
subsolution to \eqref{hjb} at $(r,x)\in\R\times\R_+$ if
any function $\psi\in C^{2,1}\Big(\R\times\R_+,\R_+\Big)$ with $\psi(r,x)=\underline{u}(r,x)$ such that $\underline{u}-\psi$ reaches the maximum at $(\bar r,\bar x)$ satisfies
\[
\mu\psi_x+a(\tilde b-r)\psi_r+\frac{\tilde \sigma^2}{2}\psi_{rr}-r\psi+\sup\limits_{0\le c\le\xi} c\big(1-\psi_x\big)\ge0
\]
and we say that a continuous function $\bar{u}:\R\times\R_+\to \R_+$ is a viscosity supersolution to \eqref{hjb3}
at $(r,x)\in\R\times\R_+$ if any function $\phi\in C^{2,1}\Big(\R\times\R_+,\R_+\Big)$ with $\phi(\bar r,\bar x)=\bar{u}(\bar r,\bar x)$ such
that $\bar{u}-\phi$ reaches the minimum at $(\bar r,\bar x)$ satisfies
\[
\mu\phi_x+a(\tilde b-r)\phi_r+\frac{\tilde\sigma^2}{2}\phi_{rr}-r\phi+\sup\limits_{0\le c\le\xi} c\big(1-\phi_x\big)\le 0\;.
\]
A viscosity solution to \eqref{hjb3} is a continuous function $u:\R\times\R_+\to \R_+$ if it is both a viscosity subsolution and a viscosity supersolution at any
$(r,x)\in \R\times \R_+$.
\end{Def}
\begin{Satz}
The value function $V(r,x)$ is a viscosity solution to \eqref{hjb3}.
\end{Satz}
\begin{proof} 
Let $(\bar r,\bar x)\in\R\times\R_+$, $\bar x>0$, $0<h<\bar x$ and $\{X_t^c\}$ the surplus process under the constant strategy $c\in[0,\xi]$. Further, we let $\tau_1:=\inf\{t\ge 0: \, X^c_t\notin\big(\bar x-h,\bar x+h\big)\}$, $\tau_2:=\inf\{t\ge 0: \, r_t\notin\big(\bar r-h,\bar r+h\big)\}$ and $\tau=\tau_1\w\tau_2$.
\\Since, the value function $V$ is locally Lipschitz continuous, there is an $n\in\N$ such that $V(r,x)-V(r_k,x_k)\le \varepsilon/2$ for $(r,x)\in [r_{k-1},r_k]\times[x_k,x_{k+1}]$, some $\varepsilon>0$ and $r_k:=\bar r-h+\frac{2h(k+1)}n$ and $x_k:=x-h+\frac{2hk}{n}$ for $k\in\N$.
Let now $C^k$ be an $\varepsilon/2$-optimal strategy for the starting point $(r_k,x_k)$. Like in Proposition \eqref{properties2}, one can show that the return function $V^{C^k}$, corresponding to the strategy $C^k$, can be applied on the initial value $(r_{\tau \w t},X_{\tau\w t}^c)$.
In particular, if $\big(r_{\tau\w t},X^c_{\tau\w t}\big)\in [r_{k-1},r_k]\times [x_k,x_{k+1}]$
\begin{align*}
V^{C^k}\big(r_{\tau\w t},X^c_{\tau\w t}\big)&\ge V^{C^k}(r_k,x_k)
\ge V(r_k,x_k)-\varepsilon/2
\ge V(r_{\tau\w t},X^c_{\tau\w t})-\varepsilon\;.
\end{align*}
Thus, for every $c\in[0,\xi]$ and a given $\varepsilon>0$ we can find a measurable strategy $C$ such that $V^{ C}\big(r_{\tau\w t},X_{\tau\w t}^c\big)\ge V\big(r_{\tau\w t},X_{\tau\w t}^c\big)-\varepsilon$.\medskip
\\At first, we show that $V$ is a supersolution.
Construct now a strategy $\tilde C=\{\tilde c_s\}$ in the following way: let $\tau $ be defined like above, $c\in[0,\xi]$ and $t\in[0,\infty)$ be fixed, define $\tilde c_s=c$ for $s\le\tau\w t$; and if $\big(r_{\tau\w t},X_{\tau\w t}^c\big)\in[r_{k-1},r_{k}]\times[x_k,x_{k+1}]$ choose from $\tau\w t$ on the strategy $C^k$, i.e. $c^k_{s-\tau\w t}=\tilde c_s$ for $s>\tau\w t$. Obviously, the constructed strategy $\tilde C$ is an admissible one.
\\Let $\phi$ be a twice continuously differentiable with respect to $r$ and once continuously differentiable with respect to $x$ test function, i.e. $V(r,x)\ge \phi(r,x)$ for all $(r,x)\in\R\times\R_+$ and $V(\bar r,\bar x)=\phi(\bar r,\bar x)$.
Since $\phi$ is smooth enough, we obtain 
\begin{equation}
\label{phi}
\begin{split}
\lim\limits_{t\to0}\mE\Big[\frac{e^{-U^{\bar r}_{\tau\w t}}\phi\big(r_{\tau\w t},x+(\mu-c)\tau\w t\big)-\phi(\bar r,\bar x)}{\tau\w t}\Big]&=(\mu-c)\phi_x(\bar r,\bar x)+a(\tilde b-\bar r)\phi_r(\bar r,\bar x)
\\&\quad{}+\frac{\tilde \sigma^2}2 \phi_{rr}(\bar r,\bar x)-\bar r\phi(\bar r,\bar x)\;.
\end{split}
\end{equation}
Further, it holds for the constructed strategy $\tilde C$:
\begin{align*}
\phi(\bar r,\bar x)=V(\bar r,\bar x)&\ge V^{\tilde C}(\bar r,\bar x)\ge c\mE\Big[\int_0^{\tau\w t} e^{-U^{\bar r}_s} \md s\Big]+\mE\Big[e^{-U_{\tau\w t}^{\bar r}}\big(V(r_{\tau\w t},X_{\tau\w t}^c)-\varepsilon\big)\Big] 
\\& \ge c \int_0^{t} \mE\big[e^{-U^{\bar r}_s}\big] \md s + \mE\Big[e^{-U_{\tau\w t}^{\bar r}}\phi(r_{\tau\w t},X_{\tau\w t}^c)\Big]-\varepsilon \mE\big[e^{-U_{\tau\w t}^{\bar r}}\big]\;.
\end{align*}
Since, the expected value $\mE\big[e^{-U_{\tau\w t}^{\bar r}}\big]$ is bounded due to the definition of $\tau$ and $\varepsilon$ was arbitrary, we have
\[
\phi(\bar r,\bar x)\ge c \int_0^{t} e^{f(r,s)} \md s + \mE\Big[e^{-U_{\tau\w t}^{\bar r}}\phi(r_{\tau\w t},X_{\tau\w t}^c)\Big]\;.
\]
In the next step, we rearrange the terms in the above inequality and divide it by $\tau\w t$. Letting $t$ go to $0$ in the above inequality yields
\[
0 \ge \mu \phi_x(\bar r,\bar x)+a(\tilde b-\bar r)\phi_r(\bar r,\bar x)+\frac{\tilde \sigma^2}{2} \phi_{rr}(\bar r,\bar x)-\bar r\phi(\bar r,\bar x)+\sup\limits_{0\le  c\le \xi}c\big(1-\phi_x(\bar r,\bar x)\big)\;,
\]
which yields the desired result.
\smallskip
\\It remains to show that $V$ is a subsolution.
Here, as usual we use the proof by contradiction. It means, we assume that $V$ is not a subsolution to \eqref{hjb3} at some $(\bar r,\bar x)$.
In particular, there is an $q>0$ and an $C^{2,1}(\R\times\R_+,\R_+)$ function $\psi_0$ such that $\psi_0(\bar r,\bar x)=V(\bar r,\bar x)$, $\psi_0(r,x)\ge V(r,x)$ for $(r,x)\in\R\times\R_+$ and $L(\psi_0)(\bar r,\bar x)<-2q$, where for some $g\in C^{2,1}\Big(\R\times\R_+,\R_+\Big)$
\begin{align*}
&L(g)(r,x):= \sup\limits_{0\le c\le\xi} \tilde L(g)(r,x)
\\&\tilde L(g)(r,x):=\mu g_x(r,x)+a(\tilde b-r)g_r(r,x)+\frac{\tilde \sigma^2}{2}g_{rr}(r,x)+c\big(1-g_x(r,x)\big)\;.
\end{align*}
Define further $\psi(r,x)=\psi_0(r,x)+ q(x-\bar x)^4+q(r-\bar r)^4$. Then, $\psi\in C^{2,1}(\R\times\R_+,\R_+)$ and $\psi(\bar r,\bar x)=V(\bar r,\bar x)$, 
\[
\psi(r,x)\ge V(r,x)+ q(x-\bar x)^4+q(r-\bar r)^4
\]
for all $(r,x)\in\R\times\R_+$. Furthermore,
\begin{align*}
L(\psi)(\bar r,\bar x)= L(\psi_0)(\bar r,\bar x)<-2q\;.
\end{align*}
Since $\psi\in C^{2,1}(\R\times \R_+,\R_+)$, the function $L(\psi)$ is continuous, such that one can find an $h>0$ with $L(\psi)(r,x)<-q$ for $(r,x)\in B_{\sqrt 2h}(\bar r,\bar x)$.
W.l.o.g. assume $\bar r>0$ and $0<h<\bar r$ and define $\Delta:=\frac{e^{(\bar r+h)h/\mu}}{\bar r-h}$ and 
\[
\varepsilon=\min\Big\{\frac{qh^4}{\Delta},q\Big\}\;.
\]
Let further $C$ be an arbitrary admissible strategy with $X^C_t=\hat X_t$, $\tau$ be defined like above.
Note, that $(r_\tau,\hat X_\tau)\in[\bar r-h,\bar r+h]\times[\bar x-h,\bar x+h]$, because the paths are continuous. Thus, we obtain
\[
V(r_\tau,\hat X_\tau)\le \psi(r_\tau,\hat X_\tau)-\Delta\varepsilon\;.
\] 
Obviously, 
\begin{align*}
L(\psi)\big(r_s,\hat X_s\big)\ge \tilde L(\psi)\big(r_s,\hat X_s\big)\;.
\end{align*}
Consider now the function $\psi$. It holds via Ito's formula
\begin{align*}
e^{-U_\tau^{\bar r}}\psi(r_\tau,\hat X_\tau)-\psi(\bar r,\bar x)&=\int_0^\tau e^{-U_s^{\bar r}}\Big\{\tilde L(\psi)\big(r_s,\hat X_s\big)-c_s\Big\}\md s
+\tilde\sigma\int_0^\tau e^{-U_s^{\bar r}}\psi_r\big(r_s,\hat X_s\big) \md W_s
\\&\le \int_0^\tau e^{-U_s^{\bar r}}L(\psi)\big(r_s,\hat X_s\big)\md s-\int_0^\tau e^{-U_s^{\bar r}}c_s\md s
\\&\quad {} + \tilde\sigma\int_0^\tau e^{-U_s^{\bar r}}\psi_r\big(r_s,\hat X_s\big) \md W_s\;. 
\end{align*}
Using $\psi(r_\tau,\hat X_\tau)\ge V(r_\tau,\hat X_\tau)+\Delta \varepsilon$ and $L(\psi)(r_s,\hat X_s)\le -\varepsilon$, we obtain
\begin{align*}
e^{-U_\tau^{\bar r}}\big(V(r_\tau,\hat X_\tau)+\Delta\varepsilon\big)-\psi(\bar r,\bar x)&\le -\varepsilon\int_0^\tau e^{-U_s^{\bar r}}\md s-\int_0^\tau c_s e^{-U_s^{\bar r}}\md s
\\&\quad{}+
 \tilde\sigma\int_0^\tau e^{-U_s^{\bar r}}\psi_r(r_s,\hat X_s) \md W_s\;.
\end{align*}
This means in particular
\begin{align*}
\int_0^\tau c_s e^{-U_s^{\bar r}}\md s+e^{-U_\tau^{\bar r}}V(r_\tau,\hat X_\tau)-\psi(\bar r,\bar x)&\le{} -\varepsilon\int_0^\tau e^{-U_s^{\bar r}}\md s-\Delta\varepsilon e^{-U_\tau^{\bar r}}
\\&\quad{}+
 \tilde\sigma\int_0^\tau e^{-U_s^{\bar r}}\psi_r(r_s,\hat X_s) \md W_s\;.
\end{align*}
Since $\psi_r(r_s,\hat X_s)$ is bounded for $s\in[0,\tau]$ and $\tau$ is a.s. finite, the stochastic integral above has expectation $0$. We can estimate the terms on the right hand side of the above inequality as follows
\begin{equation}
\label{absch:4}
\begin{split}
&\mE\Big[\int_0^\tau e^{-U_s^{\bar r}}\md s\Big]\le\mE\Big[\frac 1{\bar r-h}\big(1-e^{-(\bar r-h)\tau}\big)\Big]
\le  \frac 1{\bar r-h}\;,
\\&\mE\Big[e^{-U_\tau^{\bar r}}\Big]\ge\mE\Big[e^{-(\bar r+h)\tau}\Big]\ge e^{-(\bar r+h)\frac h \mu}\;.
\end{split}
\end{equation}
Thus, we already have shown
\[
\mE\Big[\int_0^\tau c_s e^{-U_s^{\bar r}}\md s+e^{-U_\tau^{\bar r}}V(r_\tau,\hat X_\tau)\Big]-\psi(\bar r,\bar x)\le \frac \varepsilon{\bar r-h} -\varepsilon \Delta e^{-(\bar r+h)\frac h \mu}=-\frac {2\varepsilon}{\bar r-h}\;.
\]
The same method can be applied also for $\bar r\le 0$ by just changing the estimations in \eqref{absch:4}.
\\Let $C=\{c_s\}$ be now an arbitrary admissible strategy for the starting point $(\bar r,\bar x)$, then the following estimation holds true:
\begin{align*}
V^C(\bar r,\bar x)&=\mE\Big[\int_0^\tau c_s e^{-U_s^{\bar r}}\md s+\int_\tau^\infty c_s e^{-U_s^{\bar r}}\md s\Big]
= \mE\Big[\int_0^\tau c_s e^{-U_s^{\bar r}}\md s+\int_0^\infty c_{s+\tau} e^{-U_{s+\tau}^{\bar r}}\md s\Big]
\\&\le \mE\Big[\int_0^\tau c_s e^{-U_s^{\bar r}}\md s+e^{-U_{\tau}^{\bar r}}V\big(r_\tau,X_\tau^C\big)\Big]\;.
\end{align*}
Now, we can build the supremum over all admissible strategies on the both sides of the above inequality. In particular, for every $\tilde\varepsilon>0$ there is an admissible strategy $\bar C=\{\bar c\}$ such that
\begin{align*}
\sup\limits_{C}\mE\Big[\int_0^\tau c_s e^{-U_s^{\bar r}}\md s+e^{-U_{\tau}^{\bar r}}V\big(r_\tau,X_\tau^C\big)\Big]\le \mE\Big[\int_0^\tau \bar c_s e^{-U_s^{\bar r}}\md s+e^{-U_{\tau}^{\bar r}}V\big(r_\tau,X_\tau^{\bar C}\big)\Big]+\tilde\varepsilon\;.
\end{align*}
Letting $\tilde\varepsilon=\frac{\varepsilon}{\bar r-h}$, we obtain then
\begin{align*}
V(\bar r,\bar x)- \psi(\bar r,\bar x)\le -\frac\varepsilon{\bar r-h}\;,
\end{align*}
which contradicts the assumption $\psi(\bar r,\bar x)=V(\bar r ,\bar x)$.
\end{proof}
The next result yields the uniqueness of the viscosity solution.

%% file: unique.tex
\begin{Satz}
Let $u$ be a sub- and $v$ a supersolution to HJB Equation \eqref{hjb3}, fulfilling the conditions from Proposition \ref{properties2}, \eqref{schranke} and $u(r,0)\le v(r,0)$ for all $r\in\R$. 
Then it holds $u(r,x)\le v(r,x)$ on $\R\times\R_+$. 
\end{Satz}
\begin{proof}
Assume, there is a pair $(r_0,x_0)\in \R\times\R_+$ such that $\infty>u(r_0,x_0)-v(r_0,x_0)>0$. Then, there is an $s>1$ such that for $v^s(r,x)=s v(r,x)$ it still holds $u(r_0,x_0)-v^s(r_0,x_0)>0$. The following estimation is straight forward: 
\[
u(r,x)-v^s(r,x)\le \xi\int_0^\infty e^{f(r,s)}\md s-s\xi\int_0^{\frac x{\xi-\mu}} e^{f(r,s)}\md s-s\mu\int^\infty_{\frac x{\xi-\mu}} e^{f(r,s)}\md s\;,
\] 
which means that for all $r\in\R$ there is an $\tilde x\in\R_+$ such that for $x>\tilde x$ it holds $u(r,x)-v^s(r,x)\le 0$. And on the other hand due to the properties of function $f$, for all $x\in\R_+$ one has $\lim\limits_{r\to-\infty} e^{\frac{r-b}a}\{u(r,x)-v^s(r,x)\}\le 0$.
\\Obviously, the function $v^s$ is a supersolution and using the notation from Proposition \ref{properties2} we also obtain 
\begin{align*}
u(r,x)-v^s(r,x) &= u(r,x)-u(r,0)+u(r,0)-sv(r,x)
\\&\le  \frac{x\xi e^{-\min(\frac{r-b}a,0)}}{\mu(a+b)}\big(\max(r-b,0)+\Lambda\big) + v(r, 0)-sv(r,0)
\\&\le \frac{x\xi e^{-\min(\frac{r-b}a,0)}}{\mu(a+b)}\big(\max(r-b,0)+\Lambda\big)  + (1-s)\frac{\mu}b e^{-\max(\frac{r-b}a,0)-\frac{\sigma^2}{2a^2}}\;.
\end{align*}
Assume first $r_0\le b$ and let for $r\le b$
\begin{align*}
d(r):=\frac{(s-1)(b+a)\mu^2}{b\xi \Lambda}e^{\frac{r-b}a-\frac{\sigma^2}{2a^2}} 
\quad\quad \mbox{and}\quad\quad A:=\big\{(r,x)\in\R_+^2:\;x>d(r),\, r\le b\big\}\,.
\end{align*}
Note that the function $d(r)$ is positive and increasing for $r\le b$.
As usual, we let
\[
M:=\sup\limits_{(r,x)\in A}e^{\frac{r-b}a}\{u(r,x)-v^s(r,x)\}\;.
\]
In particular, we know $\infty>M\ge e^{\frac{r_0-b}a}\{ u(r_0,x_0)-v^s(r_0,x_0)\}>0$.
Let $(r^*,x^*)$ be such that $M=u(r^*,x^*)-v^s(r^*,x^*)$ (due to the arguments above it holds $r^*>-\infty$ and $x^*<\infty$) and define for $\eta>0$ and $k:=2s\frac{\xi}{\mu(a+b)}\Lambda$ 
\begin{align*}
&H:=\{(r,q,x,y):\; d(r)<x<y,\,d(q)<y<\infty,\, -\infty<r\le b,\, r<q\le b\}\;,
\\
&f_\eta(r,q,x,y):=e^{\frac{r-b}a}u(r,x)-e^{\frac{q-b}a}v^s(q,y)- \frac\eta 2(x-y)^2-\frac{k}{\eta^2(y-x)+\eta}\;,
\\&M_\eta:=\sup\limits_{(r,q,x,y)\in H}f_\eta(r,q,x,y)\;.
\end{align*}
Note that $f_\eta$ is continuous, which guarantees the existence of $(r_\eta,q_\eta,x_\eta,y_\eta)\in\bar H$, where $\bar H$ denotes the closure of $H$, such that $M_\eta=f_\eta(r_\eta,q_\eta,x_\eta,y_\eta)$. By definition of $(r^*,x^*)$ it holds $(r^*,r^*,x^*,x^*)\in \bar H$. Thus,
\[
M_\eta\ge f_\eta(r^*,r^*,x^*,x^*)=e^{\frac{r^*-b}a}\big(u(r^*,x^*)-v^s(r^*,x^*)\big)-\frac{k}\eta=e^{\frac{r^*-b}a}M-\frac{k}\eta\;. 
\]
We can therefore conclude that there is an $\eta^*$ such that $M_\eta>0$ for all $\eta>\eta^*$ and $\liminf\limits_{\eta\to\infty}M_\eta\ge e^{\frac{r^*-b}a} M$. 
Further, it is clear that because $v^s$ is bounded in $y$ it holds
$\lim\limits_{y\to\infty}f_\eta(r,q,x,y)=-\infty$. \smallskip
\\Obviously, $f_\eta$ is decreasing in $q$, which means that we can assume $r_\eta=q_\eta$, i.e. we consider $(r,r,x,y)\in\bar H$.
For $(r,r,x,x)\in\bar H$ and $h>0$ we have
\begin{align*}
\limsup\limits_{h\to0}\frac{f_\eta(r,r,x,x)-f_\eta(r,r,x,x+h)}{h}\le s\frac{\xi}{\mu(a+b)}\Lambda-k<0\;,
\end{align*}
Thus, there is an $\varepsilon_1>0$ such that $f_\eta(r,r,x,y)>f_\eta(r,r,x,x)$ for $y\in(x,x+\varepsilon_1]$ and $x\in[d(r),\infty)$. For $y\ge \varepsilon_1 +x$ one has, independent of the values of $x$ and $y$:
\begin{align*}
f_\eta(r,r,x,y)&=u(r,x)-v^s(r,y)-\frac\eta 2(x-y)^2-\frac{k}{\eta^2(y-x)+\eta}
\\&\le (\xi-s\mu)\int_0^\infty e^{f(r,s)}\md s-\frac\eta 2\varepsilon_1^2 <0
\end{align*}
for $\eta> \frac{\xi-s\mu}{\varepsilon_1}\int_0^\infty e^{f(r,s)}\md s$.
Thus, $f_\eta(r,r,x,x)\le f_\eta(r,r,x,x+\varepsilon)<0$.\medskip
\\Letting 
\begin{align*}
&d(r):=\frac{(s-1)(b+a)\mu^2}{b\xi (r-b+\Lambda)}e^{-\frac{r-b}a-\frac{\sigma^2}{2a^2}}
\\&H:=\{(r,q,x,y):\; d(r)<x<y,\,d(q)<y<\infty,\, b<r<\infty,\, b<q<r\}
\end{align*}
one can show the uniqueness also for $r>b$.
\end{proof}
\begin{Bem}
The problem with a deterministic linear surplus and an Ornstein-Uhlenbeck process as a short rate seemed to be very simple. Nevertheless, we could not find an explicit solution to this optimization problem. The value function has been proved to be concave in $r$ and convex in $x$, which suggests that the optimal consumption strategy should be of a barrier type. But in contrast to the case with a geometric Brownian motion as a discounting factor, it is not that easy to calculate the return functions corresponding to some barrier strategy.
\end{Bem}